\documentclass[11pt]{amsart}

\usepackage{mathtools,amssymb,amsthm}
\usepackage{hyperref}

\newtheorem{theorem}{Theorem}

\newtheorem{lemma}[theorem]{Lemma}

\theoremstyle{definition}

\theoremstyle{remark}

\title{Explicit Measurable Counterexamples to Finite-Difference Operators}
\author{Hajrudin Fejzi\'c}
\date{}

\begin{document}

\begin{abstract}
Let
\[
D(f,h)=\sum_{k=1}^n a_k f(b_kh)
\]
be a finite-difference operator, where the coefficients may be complex
and the nodes are real. We study the asymptotic implication
\[
D(f,h)=o(h)\quad\Longrightarrow\quad f(h)=o(h)
\]
for Lebesgue measurable functions. In 2017, Ash, Catoiu, and
Cs\"ornyei \cite{ACC} showed that this implication holds for every
measurable function precisely for operators belonging to a very
restricted class of finite differences.

To establish that every operator outside the exceptional class admits a
counterexample, Ash, Catoiu, and Cs\"ornyei analyze an infinite system
of linear equations. Their argument is therefore existential in nature
and does not provide an explicit description of the counterexamples. By
contrast, we give a substantially simpler and elementary construction
based on zeros of finite Laurent polynomials and multiplicative
characters on finitely generated multiplicative subgroups. This provides a direct recipe for
constructing measurable functions \(f\) such that
\[
D(f,h)=0\qquad\text{for every }h\in\mathbb C,
\]
while, in fact,
\[
f(h)\not=o(1)\qquad\text{as }h\to0.
\]

\end{abstract}

\maketitle

Let \(n\in\mathbb N\), let \(a_1,\ldots,a_n\in\mathbb C\), and let
\(b_1,\ldots,b_n\in\mathbb R\) be distinct. For a function \(f\) defined
on \(\mathbb C\), we consider the finite-difference operator
\[
D(f,h)=\sum_{k=1}^n a_k f(b_kh).
\]
The numbers \(b_k\) are called the nodes, and the non-zero numbers \(a_k\) are the
corresponding coefficients.

This paper concerns the implication
\[
D(f,h)=o(h)\quad\text{as }h\to0
\qquad\Longrightarrow\qquad
f(h)=o(h)\quad\text{as }h\to0
\]
for measurable functions \(f\). In the real-valued setting, Ash, Catoiu,
and Cs\"ornyei \cite{ACC} completely characterized the finite differences
for which this implication holds for every measurable function
\(f\colon\mathbb R\to\mathbb R\). Surprisingly, apart from multiplication
by a nonzero constant and a rescaling of the variable, the only
exceptional difference is
\[
R(f,h)=\lambda f(\beta h)+\lambda f(-\beta h)
       +f(h)-f(-h),
\qquad \lambda\beta\ne0.
\]

The difficult part of their result is the converse. Namely, if a finite
difference is not of this exceptional form, then there exists a measurable
function \(f\) such that
\[
D(f,h)=0\qquad\text{for every }h,
\]
while
\[
f(h)\not=o(h)\qquad\text{as }h\to0.
\]

The proof in \cite{ACC} is existential in nature: the counterexample is
obtained from the solvability of an infinite system of linear equations.
In contrast, our proof gives an explicit construction. Crucially, our framework dramatically simplifies this problem: rather than managing vast infinite systems of equations, the counterexamples can be written down almost immediately by inspection.  They
are supported on countable multiplicative subgroups and are described by
multiplicative characters whose parameters are obtained from zeros of
finite Laurent polynomials. 

To demonstrate the remarkable ease and algorithmic nature of this approach, we show how a concrete counterexample can be effortlessly constructed from scratch using a simple algebraic recipe. Consider the operator
\[
D(f,h)=f(4h)-3f(6h)+f(9h)
      =f(2^2h)-3f(2\cdot3h)+f(3^2h).
\]
The corresponding Laurent polynomial is
\[
P_+(x,y)=x^2-3xy+y^2.
\]

In order to find a nonzero root of $P_+$, we first choose an integer point that does not
belong to any of the hyperplanes determined by the exponent vectors. In
this case, the relevant hyperplanes are
\[
2u=u+v,\qquad 2u=2v,\qquad u+v=2v.
\]
The integer point \((1,2)\) avoids all three of them. Consequently,
\[
P_+(t,t^2)=t^2-3t^3+t^4
          =t^2(t^2-3t+1).
\]
Setting
$
t=\frac{3+\sqrt5}{2},
$
we see that \(t\) is a nonzero root of \(P_+(t,t^2)\).

Now choose the countable multiplicative subgroup
\[
G=
\left\{
\varepsilon 5^\ell 2^p3^q:
\varepsilon\in\{\pm1\},\quad
\ell,p,q\in\mathbb Z
\right\},
\]
generated by \(-1,2,3\), together with the auxiliary prime \(5\).
Since \(2,3,\) and \(5\) are multiplicatively independent, every element
of \(G\) has a unique representation of this form. 
The auxiliary generator \(5\) is used only to provide a sequence
\(5^{-\ell}\to0\); the definition of \(f\) does not depend on the
exponent \(\ell\).

Define
\(f\colon\mathbb C\to\mathbb R\) by setting
\(
f(z)=0\) if \(z\notin G,
\)
and
\(
f\left(\varepsilon 5^\ell2^p3^q\right)
=
t^{p+2q}.
\)
A direct calculation yields
\[
f(4h)=t^2f(h),\qquad
f(6h)=t^3f(h),\qquad
f(9h)=t^4f(h)
\]
for every \(h\in\mathbb C\). Hence
\[
D(f,h)
=
\bigl(t^2-3t^3+t^4\bigr)f(h)=P_+(t,t^2)f(h)=0
\]
for every \(h\in\mathbb C\). Since \(G\) is countable, \(f\) is trivially Lebesgue
measurable. Moreover, because
\[
f(5^{-\ell})=1\qquad(\ell\in\mathbb N),
\]
it follows that \(f\not=o(1)\) as \(h\to0\), completing the counterexample.
\begin{lemma}\label{Laurent}
Let

\[
P(x_1,\dots,x_m)
=
\sum_{k=1}^n
a_k
x_1^{\alpha_{k1}}\cdots x_m^{\alpha_{km}}
\]
be a Laurent polynomial with non-zero coefficients and at least two distinct exponent vectors, 
$$(\alpha_{i1}, \ldots, \alpha_{im}) \ne (\alpha_{j1}, \ldots, \alpha_{jm}).$$

Then there exists
$
A\in (\mathbb{C}\setminus\{0\})^m
$
such that
$
P(A)=0.
$
\end{lemma}

\begin{proof}

By combining like terms, we may assume that all of the exponent vectors of $P$ are distinct.
For each pair of exponent vectors,
$$(\alpha_{i1}, \ldots, \alpha_{im}) \ne (\alpha_{j1}, \ldots, \alpha_{jm}),$$ 
consider the hyperplane in $\mathbb{R} ^m$ defined by 
$$
\alpha_{i1}x_1+ \cdots + \alpha_{im}x_m = \alpha_{j1}x_1+ \cdots + \alpha_{jm}x_m.
$$
Since a finite union of hyperplanes cannot cover $\mathbb{R}^m$, nor can it cover the lattice $\mathbb{Z}^m$, there exists an integer lattice  point $r= (r_1, \ldots, r_m)\in \mathbb{Z}^m$ that does not lie on any of the hyperplanes.

Define
\[
Q(t)=P(t^{r_1},\dots,t^{r_m}).
\]
Then \(Q\) is a nonconstant Laurent polynomial in one variable. 
By choice of $r$ all of the exponents of $Q$ are distinct.
Multiplying by a suitable power of \(t\), we obtain a nonconstant ordinary polynomial
\[
R(t)=t^NQ(t).
\]
Since \(R\) has at least two terms, it is not a monomial. Hence, there exists
$
a\in\mathbb{C}\setminus\{0\}
$
such that
$
R(a)=0.
$
Therefore
$
Q(a)=0,
$
and consequently
\[
P(a^{r_1},\dots,a^{r_m})=0.
\]
Since \(a\neq0\), all coordinates \(a^{r_i}\) are nonzero.
\end{proof}

\begin{theorem}
Let
\[
D(f,h)=\sum_{k=1}^n a_k f(b_kh)
\]
be a finite difference with nonzero coefficients \(a_k\in\mathbb C\) and
distinct nodes \(b_k\in\mathbb R\). Suppose that there are at least two
nonzero nodes and that \(D\) cannot be scaled to be of the form
\[
R(f,h)=af(bh)+af(-bh)+f(h)-f(-h).
\]
Then there exists a Lebesgue measurable function
$
f\colon\mathbb C\longrightarrow\mathbb C
$
such that
$
f(h)\neq o(1)$ as $h\to0
$
and
\[
D(f,h)=0\qquad\text{for every }h\in\mathbb C.
\]
If the coefficients of $D$ are real, then either the real part or
the imaginary part of $f$ provides a
real-valued function with the same properties.
\end{theorem}

\begin{proof}
If one of the nodes is zero, say \(b_1=0\), define
\[
f(z)=
\begin{cases}
-\dfrac{1}{a_1}\displaystyle\sum_{k=2}^n a_k,
   & z=0,\\[6pt]
1, & z\neq0.
\end{cases}
\]
Then \(f(h)\neq o(1)\) as \(h\to0\), and
\[
D(f,h)=a_1f(0)+\sum_{k=2}^n a_k=0
\]
for every \(h\in\mathbb C\). We may therefore assume that none of the nodes
\(b_k\) is zero.

Let \(H\) be the multiplicative subgroup of
\(\mathbb R\setminus\{0\}\) generated by
\[
\{-1,b_1,\dots,b_n\}.
\]
Since \(H\) is finitely generated, there exist positive real numbers
\(q_1,\dots,q_m\) such that every element of \(H\) can be written uniquely
in the form
\[
\varepsilon q_1^{n_1}\cdots q_m^{n_m},
\qquad
\varepsilon\in\{\pm1\},\quad n_i\in\mathbb Z.
\]
For each node \(b_k\), write
\[
b_k=\sigma_kq_1^{\alpha_{k1}}\cdots q_m^{\alpha_{km}},
\qquad
\sigma_k\in\{\pm1\}.
\]

Define the Laurent polynomials
\[
P_+(x_1,\dots,x_m)
=
\sum_{k=1}^n
a_kx_1^{\alpha_{k1}}\cdots x_m^{\alpha_{km}}
\]
and
\[
P_-(x_1,\dots,x_m)
=
\sum_{k=1}^n
a_k\sigma_kx_1^{\alpha_{k1}}\cdots x_m^{\alpha_{km}}.
\]

Suppose first that, after combining like terms and eliminating terms with
zero coefficients, \(P_+\) is not a monomial. By
Lemma~\ref{Laurent}, there exists
\[
A=(A_1,\dots,A_m)\in(\mathbb C\setminus\{0\})^m
\]
such that
\[
P_+(A)=0.
\]

Choose \(s>1\) such that
\[
\log s,\log q_1,\dots,\log q_m
\]
are linearly independent over \(\mathbb Q\), and let
\[
G=
\left\{
\varepsilon s^\ell q_1^{n_1}\cdots q_m^{n_m}:
\varepsilon\in\{\pm1\},\
\ell,n_1,\dots,n_m\in\mathbb Z
\right\}.
\]
For each element of \(G\) all the integers $\varepsilon, \ell, n_1,\ldots, n_m$ are unique.

Define \(f\colon\mathbb C\to\mathbb C\) by
\[
f(z)=0\qquad\text{if }z\notin G,
\]
and
\[
f\left(
\varepsilon s^\ell q_1^{n_1}\cdots q_m^{n_m}
\right)
=
A_1^{n_1}\cdots A_m^{n_m}.
\]

For every \(z\in\mathbb C\) and every node \(b_k\), we have
\[
f(b_kz)
=
A_1^{\alpha_{k1}}\cdots A_m^{\alpha_{km}}f(z).
\]
Indeed, this follows directly from the definition when \(z\in G\); when
\(z\notin G\), both sides vanish because \(b_k\in G\) and hence
\(b_kz\notin G\). Consequently,
\[
D(f,z)
=
P_+(A)f(z)
=
0
\]
for every \(z\in\mathbb C\).

Suppose next that, after combining like terms and eliminating terms with
zero coefficients, \(P_-\) is not a monomial. Choose
\[
A\in(\mathbb C\setminus\{0\})^m
\]
such that
\[
P_-(A)=0.
\]
Using the same group \(G\), define \(f(z)=0\) for \(z\notin G\), and define
\[
f\left(
\varepsilon s^\ell q_1^{n_1}\cdots q_m^{n_m}
\right)
=
\varepsilon A_1^{n_1}\cdots A_m^{n_m}.
\]
Then 
\[
f(b_kz)
=
\sigma_k
A_1^{\alpha_{k1}}\cdots A_m^{\alpha_{km}}f(z)
\]
for every \(z\in\mathbb C\). It follows that
\[
D(f,z)=P_-(A)f(z)=0
\]
for every \(z\in\mathbb C\).

In either case, along the sequence
\[
h_\ell=s^{-\ell},
\]
we have \(h_\ell\to0\) and
\[
|f(h_\ell)|=1.
\]
Thus \(f(h)\neq o(1)\) as \(h\to0\). Moreover, \(G\) is countable, so
\(f\), being zero outside the countable set \(G\), is Lebesgue measurable
on \(\mathbb C\).

It remains to consider the case in which both \(P_+\) and \(P_-\)
reduce to monomials or to the zero polynomial.

Suppose first that \(P_+\) is a monomial. Then there is a node \(b\) such that, if \(-b\) is also a node, the sum
of the corresponding coefficients is nonzero. Moreover, for every node
\(c\notin\{b,-b\}\), the node \(-c\) is also present, and the
coefficients of \(c\) and \(-c\) are negatives of one another.

If \(P_+\) is the zero polynomial, then every node occurs together with
its opposite, and the coefficients of \(c\) and \(-c\) are negatives of
one another.

Similarly, if \(P_-\) is a monomial, then there is a node \(b\) such that, if \(-b\) is also a node, the difference
of the corresponding coefficients is nonzero. Moreover, for every node
\(c\notin\{b,-b\}\), the node \(-c\) is also present, and the
coefficients of \(c\) and \(-c\) are the same.

The two polynomials cannot both be zero, since this would force every
coefficient to be zero. If \(P_-\) is the zero polynomial, then
\[
D(f,h)=a f(bh)+a f(-bh).
\]
 Define
\[
f(z)=
\begin{cases}
1,&z>0,\\
-1,&z<0,\\
0,&\text{otherwise}.
\end{cases}
\]
Then \(D(f,h)=0\) for every \(h\), and \(f(h)\not=o(1)\) as \(h\to0\).

If \(P_+\) is the zero polynomial, then
\[
D(f,h)=a f(bh)-a f(-bh).
\]
 In this case the constant function \(f\equiv1\)
satisfies \(D(f,h)=0\) for every \(h\).

Finally, suppose that both \(P_+\) and \(P_-\) are monomials. Then,
after relabeling the two pairs if necessary,
\[
D(f,h)
=
a_bf(bh)+a_bf(-bh)
+a_cf(ch)-a_cf(-ch).
\]
After multiplying the difference by a nonzero constant and rescaling the variable $h$, this reduces precisely to the  exceptional form
\[
R(f,h)=af(bh)+af(-bh)+f(h)-f(-h),
\]
contrary to the hypothesis.
\end{proof}

The characterization of measurable finite differences obtained by
Ash, Catoiu, and Cs\"ornyei becomes substantially richer under the
additional assumption that \(f(h)=o(1)\) as \(h\to0\). For example, the
finite difference
\[
D(f,h)=f(2h)-3f(h)
\]
has the property that
\[
D(f,h)=o(h)
\quad\text{and}\quad
f(h)=o(1)
\]
imply
\[
f(h)=o(h).
\]
This and related applications to Riemann derivatives are established in \cite{HF}.


\begin{thebibliography}{10}

\bibitem{ACC}
J. M. Ash, S. Catoiu, and M. Cs\"ornyei,
\textit{Generalized vs. ordinary differentiation},
Proc. Amer. Math. Soc. 145 (2017), 1553--1565.

\bibitem{HF}
H. Fejzi\'c,
\textit{The Marcinkiewicz--Zygmund Property for Riemann Differences with Geometric Nodes},
Real Anal. Exchange 50(2) (2025), 445--468.


\end{thebibliography}
\end{document}